\documentclass[11pt,a4paper]{article}
\usepackage{amsmath,amsthm,amsfonts,amssymb}
\usepackage{mathrsfs}
\usepackage[hmargin=1in,vmargin=1in]{geometry}
\usepackage[authoryear]{natbib}
\usepackage{url}
\usepackage{lineno}
\usepackage{xcolor}
\usepackage{graphicx}

\numberwithin{equation}{section}

\newtheorem{thm}{Theorem}[section]
\newtheorem{prop}[thm]{Proposition}

\theoremstyle{remark}
\newtheorem{rem}[thm]{Remark}

\theoremstyle{definition}

\newtheorem{ex}[thm]{Example}

\newcommand{\D}[4]{{}^{#1}_{#2}D_{#3}^{#4}}
\renewcommand{\L}{\mathcal L}
\newcommand{\N}{\mathbb N}
\newcommand{\R}{\mathbb R}
\newcommand{\C}{\mathbb C}

\newcommand{\ceil}[1]{\lceil #1 \rceil}
\renewcommand{\d}{\mathrm d}
\newcommand{\e}{\mathrm e}

\DeclareMathOperator*{\erf}{erf}
\DeclareMathOperator*{\erfc}{erfc}

\begin{document}

\title{Analytical solutions of moving boundary problems for the time-fractional diffusion equation}
\author{M. Rodrigo\thanks{School of Mathematics and Applied Statistics,
University of Wollongong, Wollongong, New South Wales, Australia.
E-mail: {\tt mrodrigo-mail@gmx.com}}}
\date{21 September 2022}
\maketitle

\begin{abstract}

\noindent
The time-fractional diffusion equation is considered, where the time derivative is either of Caputo or Riemann-Liouville type. The solution of a general initial-boundary value problem with time-dependent boundary conditions over bounded and unbounded domains is derived using the embedding method. The solution of the initial-boundary value problem, expressed in terms of a two-parameter auxiliary function, is used to obtain analytical solutions of moving boundary problems. In particular, a `fractional' analogue of the Neumann solution to a classical Stefan problem for melting ice is found.


\medskip
\noindent {\bf MSC~2020 Subject Classifications:} 26A33; 35R11; 35K05


\end{abstract}

\section{Introduction}

Suppose that we are given the following partial differential equation~(PDE) for $u(x,t)$:
\begin{equation}
\label{diffusion-wave-PDE}
\D{}{}{}{2 \nu} u = \kappa \frac{\partial^2 u}{\partial x^2}, \quad x \in \R, \quad t > 0,
\end{equation}
where $\kappa > 0$ and $0 < \nu \le 1$. The `time-fractional derivative operator'~$\D{}{}{}{2 \nu}$ is such that \eqref{diffusion-wave-PDE} reduces to the diffusion equation and the wave equation when $\nu = \frac{1}{2}$ and $\nu = 1$, respectively. The behaviour of a solution of \eqref{diffusion-wave-PDE} is said to be `diffusion-like' (respectively, `wave-like') when $0 < \nu \le \frac{1}{2}$ (respectively, $\frac{1}{2} < \nu \le 1$) and we refer to \eqref{diffusion-wave-PDE} as the time-fractional diffusion equation (respectively, time-fractional wave equation). 

The definition of $\D{}{}{}{2 \nu}$ relies on certain concepts from the field of mathematics known as the fractional calculus~\citep{MiRo93,Po99}. For notational convenience, define the function
\begin{equation}
\label{delta-fun}
\delta_\mu(t) = 
\begin{cases}
\frac{t^{\mu - 1}}{\Gamma(\mu)} & \text{if $\mu > 0$}, \\
\delta(t) & \text{if $\mu = 0$}, 
\end{cases}
\end{equation}
where $\Gamma(\mu)$ is the Euler gamma function and $\delta(t)$ is the Dirac delta function. For a suitable function~$y(t)$, the Riemann-Liouville fractional integral of order~$\mu$ is 
$$
\D{}{0}{t}{-\mu} y(t) = \frac{1}{\Gamma(\mu)} \int_0^t (t - \tau)^{\mu - 1} y(\tau) \, \d \tau.
$$
This can be expressed as the Laplace convolution
\begin{equation}
\label{conv-int}
\D{}{0}{t}{-\mu} y(t) = (\delta_\mu * y)(t).
\end{equation} 
Let $\ceil{\mu}$ denote the least integer greater than or equal to $\mu$, so that $\ceil{\mu} \ge \mu$. The Caputo fractional derivative of order~$\mu$ is defined as
$$
\D{C}{0}{t}{\mu} y(t) =  \D{}{0}{t}{-(\ceil{\mu} - \mu)}  D^{\ceil{\mu}}  y(t),
$$
whereas the Riemann-Liouville fractional derivative of order~$\mu$ is given by
$$
\D{}{0}{t}{\mu} y(t) = D^{\ceil{\mu}} \D{}{0}{t}{-(\ceil{\mu} - \mu)} y(t).
$$
Here, $\D{}{0}{t}{-(\ceil{\mu} - \mu)}$ is a Riemann-Liouville fractional integral operator and $D^{\ceil{\mu}}$ is an ordinary derivative operator. When $\mu = m \in \N$, the Riemann-Liouville fractional integral reduces to $m$-fold integration, while the Caputo and Riemann-Liouville fractional derivatives simplify to $m$-fold differentiation. 

This article considers initial-boundary value problems~(IBVPs) and moving boundary problems associated with \eqref{diffusion-wave-PDE} both when $\D{}{}{}{2 \nu} = \D{C}{0}{t}{2 \nu}$ and $\D{}{}{}{2 \nu} = \D{}{0}{t}{2 \nu}$, where $0 < \nu \le \frac{1}{2}$. Note that if $n \in \N$, then $D^n u(x,t)$ refers to the $n$th partial derivative of $u(x,t)$ with respect to $t$. 

The Caputo time-fractional diffusion equation (i.e.~$\D{}{}{}{2 \nu}= \D{C}{0}{t}{2 \nu}$ and $0 < \nu \le \frac{1}{2}$) was used by \citet{Ni86} to model diffusion in media with fractal geometry. More recently, using a Caputo time-fractional diffusion equation, \citet{WeChZh15} developed a model to describe how chloride ions penetrate reinforced concrete structures exposed to chloride environments.

Moving boundary problems arise in many areas of science and engineering~\citep{Cr84,Hi87,Gu03}. Some applications include modelling of biological and tumour invasion~\citep{CrGu72,ElMcSi20}, drug delivery~\citep{SaMaHa17} and melting of crystal dendrite~\citep{MoKiMoMc19}. The classical one-dimensional Stefan problem is a canonical moving boundary problem that models the melting of ice; see some historical notes in \citet{Vu93}. In this context, the PDE is referred to as the heat equation instead of the diffusion equation. Since Stefan's seminal work, moving boundary problems have been extensively studied. Excellent surveys can be found in the books by \cite{Cr84,Hi87,Gu03} and the references therein. 

As moving boundary problems are typically nonlinear, they are usually studied using numerical and approximate analytical methods. \citet{Fu80} performed a comparison of different numerical methods for moving boundary problems; see also \citet{CaKw04,LeBaLa15} for a study of numerical methods for one-dimensional Stefan problems. Approximate analytical methods for one-dimensional Stefan problems include the heat balance integral method~\citep{Go58,MiMy08,MiMy12}, the refined integral method~\citep{SaSiCo06} and the combined integral method~\citep{MiMy11,Mi12}.

Exact analytical solutions of some one-dimensional Stefan problems are reviewed in \citet{Cr84,Hi87}. However, such solutions of moving boundary problems are quite rare because these problems are highly nonlinear. Hence standard methods for linear problems such as separation of variables, Green's functions and integral transforms are usually not applicable. \citet{RoTh21} used the embedding method to find exact analytical solutions of one-dimensional moving boundary problems for the heat equation. They also showed how the embedding method can be adapted to two-phase Stefan problems. 

In fact, \citet{RoTh21} considered a general IBVP for the heat equation with time-dependent boundary condition~(BCs) and derived the analytical solution using an embedding technique. The same technique is able to handle both bounded and unbounded spatial domains, unlike the standard solution techniques mentioned above. More recently, \cite{RoTh22} studied a diffusion-advection-reaction equation and solved the associated IBVP analytically with the embedding method and proposed a numerical method for solving systems of linear Volterra integral equations of the first kind that naturally arise from the technique. The embedding method was introduced in \citet{Ro14} in the context of pricing American call and put options, and was subsequently adapted to price barrier options~\citep{GuRoSa20} and perpetual American options with general payoffs~\citep{Ro22a}. 

In many applications of diffusion-advection-reaction equations to model contaminant or solute transport in porous media, the boundaries are usually assumed to be constant in time. However, solute transport problems can involve various types of time-dependent BCs~\citep{NgRiSt88,HoGeLe00,GaFuZhMa13}. The application of the embedding method to multilayer diffusion problems with time-dependent BCs is the subject of a recent article~\citep{Ro22c}.

\citet{Ro22b} extended the embedding technique to propose a unified way to solve initial value problems~(IVPs) and IBVPs for the time-fractional diffusion-wave equation~\eqref{diffusion-wave-PDE} (i.e.~$0 < \nu \le 1$). The class of IBVPs considered was limited to those with spatial domains where $0 \le x < \infty$ and with Dirichlet-type (time-constant) BCs imposed at $x = 0$. The first contribution of the present article is to generalise the results in \citet{Ro22b} by solving IBVPs for the time-fractional diffusion equation (i.e.~$0 < \nu \le \frac{1}{2}$) with general time-dependent BCs over bounded and unbounded domains, similar to what was done in \citet{RoTh21} for the classical diffusion equation. The second contribution of the present article is to use the generalisation to find analytical solutions of moving boundary problems for the time-fractional diffusion equation. The reason for the restriction~$0 < \nu \le \frac{1}{2}$, instead of $0 < \nu \le 1$, is because we wish to consider `fractional Stefan problems' in this article. Hence we have to restrict to moving boundary problems whose solutions have `diffusion-like behaviour'. 

The formulation of Stefan problems for the heat equation includes an extra condition (known as the Stefan condition) that prescribes the dynamics for the unknown moving boundary. As we will consider the time-fractional diffusion equation here, the Stefan condition will be replaced by an analogous `fractional Stefan condition'. However, it is important to point out that the physical motivation for considering moving boundary problems (in fact, IBVPs in general) for the time-fractional diffusion equation remains an open problem. In this article, we approach the study of such problems from a theoretical viewpoint. 

The outline of this article as follows. In Section~2, we revisit a two-parameter auxiliary function introduced in \citet{Ro22b} by first summarising some of its properties and then deriving new properties that will be especially relevant for moving boundary problems. In Section~3, we formulate a general IBVP for the time-fractional diffusion equation and obtain the solution using the embedding method. Section~4 studies moving boundary problems via two illustrative examples, one with a bounded domain and the other with an unbounded domain. Brief concluding remarks are given in Section~5.

\section{A useful auxiliary function and its properties}

In this section, we investigate some properties of an auxiliary function that are useful in the study of the time-fractional diffusion-wave equation. 

%
%

\subsection{Summary of known properties of the auxiliary function}

Let $\mu \ge 0$, $0 < \nu \le 1$ and $a > 0$. \citet{Ro22b} defined the function
\begin{equation}
\label{R-def}
R_{\mu,\nu}(a,t) = \L^{-1}\{s^{-\mu} \e^{-a s^\nu};t\}
\end{equation}
as an inverse Laplace transform. Since $\L\{\D{}{0}{t}{-\mu}f(t);s\} = s^{-\mu} \L\{f(t);s\}$, we deduce that 
\begin{equation}
\label{R-basic}
R_{\mu,\nu}(a,t) = {}_{0}^{}D_{t}^{-\mu} R_{0,\nu}(a,t)
\end{equation} 
and thus $R_{0,\nu}(a,t)$ can be interpreted as more `basic' than $R_{\mu,\nu}(a,t)$. For the convenience of the reader, in this subsection, we summarise some of the properties of $R_{\mu,\nu}(a,t)$ that were proved in \citet{Ro22b}.

The function~$y(t) = R_{\mu,\nu}(a,t)$ verifies $y(0+) = 0$ and satisfies the fractional integral equation
\begin{equation}
\label{R-int-eq}
a \nu \D{}{0}{t}{-(1 - \nu)} y(t) = t y(t) - \mu \int_0^t y(\tau) \, \d \tau
\end{equation}
and the fractional ordinary differential equation
\begin{equation}
\label{R-diff-eq}
a \nu \D{}{0}{t}{\nu} y(t) = a \nu \D{C}{0}{t}{\nu} y(t) = t y'(t) + (1 - \mu) y(t).
\end{equation}
To evaluate $R_{\mu,\nu}(t)$, we can either perform a numerical Laplace transform inversion in \eqref{R-def} or implement finite difference schemes to solve the integral equation~\eqref{R-int-eq} or the differential equation~\eqref{R-diff-eq}. For example, numerical Laplace transform inversion was used to obtain profiles of $R_\nu(2.5,t)$, as shown in Figure~1 for $\nu = 0.3, 0.4, 0.5, 0.6, 0.7$.
\begin{figure}[ht]
\label{R-plot}
\centering
\includegraphics[scale=0.3]{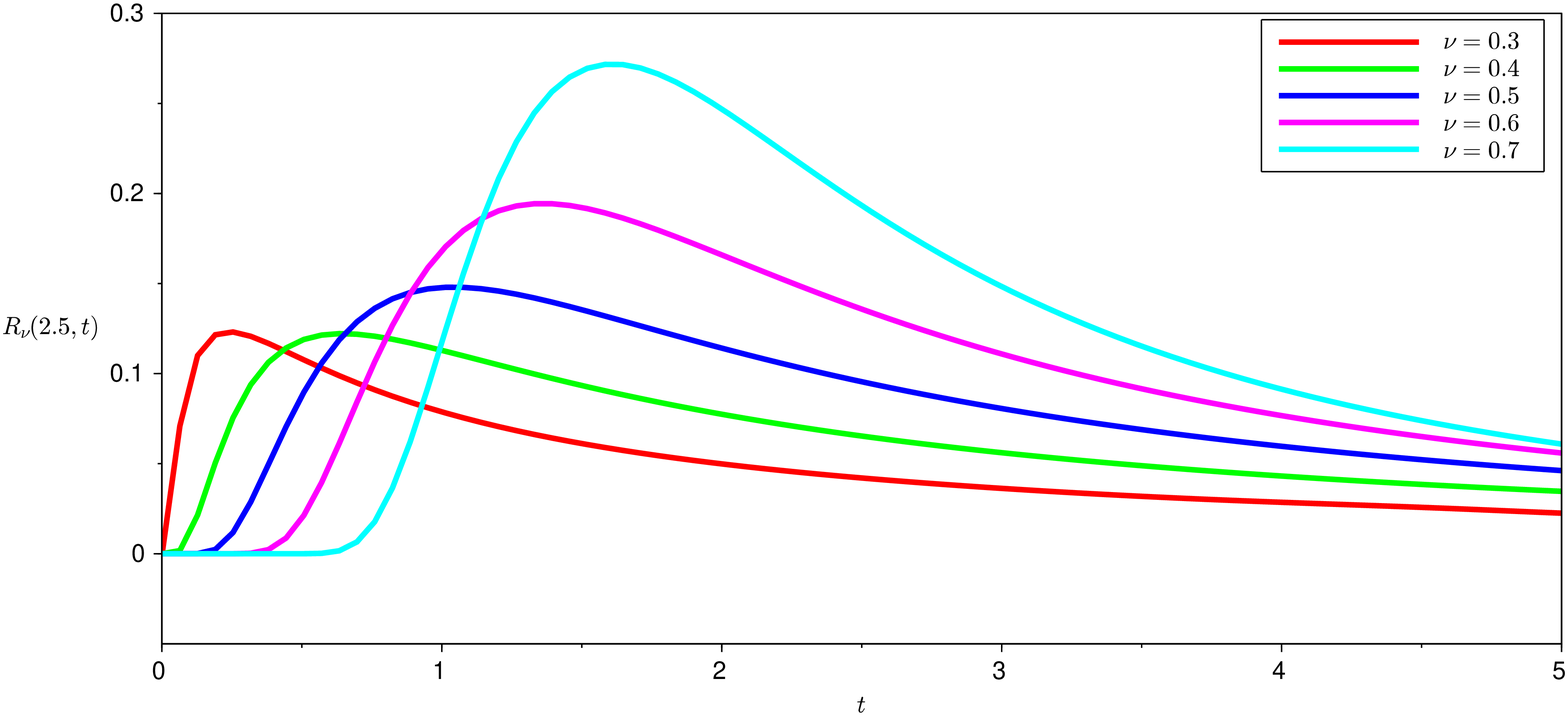}
\caption{Plot of $R_\nu(2.5,t)$ for different values of $\nu$.}
\end{figure}

When $\mu = 0$, $0 < \nu \le \frac{1}{2}$ and $a > 0$, an alternative integral representation of \eqref{R-def} is
\begin{equation}
\label{R-alt}
R_{0,\nu}(a,t) = \frac{1}{\pi} \int_0^\infty \e^{-t z} \e^{-a \cos(\pi \nu) z^\nu} \sin(a \sin(\pi \nu) z^\nu) \, \d z.
\end{equation}
An analogous integral representation when $\mu > 0$, $0 < \nu \le \frac{1}{2}$ and $a > 0$ can be obtained using \eqref{R-basic} in \eqref{R-alt} and taking the Riemann-Liouville fractional integral of the exponential function~$t \mapsto \e^{-t z}$. Note, however, that \eqref{R-alt} is not necessarily valid when $\frac{1}{2} < \nu \le 1$~\citep{Ro22b}. 

If $\mu \ge 0$, $0 < \nu \le 1$ and $a > 0$, then
\begin{equation}
\label{R-int-prop-1}
R_{\mu + \nu,\nu}(a,t) = \int_a^\infty R_{\mu,\nu}(z,t) \, \d z.
\end{equation}
In particular, $\mu = \nu$ gives
\begin{equation}
\label{R-int-prop-2}
R_{2 \nu,\nu}(a,t) = \int_a^\infty R_{\nu,\nu}(z,t) \, \d z.
\end{equation}
Some special cases are
\begin{equation}
\label{R-special-cases}
R_{0,\frac{1}{2}}(a,t) = \frac{a \e^{-\frac{a^2}{4 t}}}{2 \sqrt{\pi t^3}}, \quad R_{\frac{1}{2},\frac{1}{2}}(a,t) = \frac{\e^{-\frac{a^2}{4 t}}}{\sqrt{\pi t}}, \quad R_{1,\frac{1}{2}}(a,t) = \erfc\Big(\frac{a}{2 \sqrt{t}}\Big),
\end{equation}
which follow from \eqref{R-alt}, \eqref{R-basic} and \eqref{R-int-prop-2}, respectively.

\subsection{Further properties of the auxiliary function}

Here, we derive new properties of the auxiliary function that are needed for solving IBVPs for the time-fractional diffusion equation.

In the previous subsection, it was pointed out that $R_{\mu,\nu}(a,0+) = 0$ for a fixed~$a$. The following result derives a similar property for $R_{\mu,\nu}(0+,t)$ with $t$ fixed.
\begin{prop}
\label{R-a-zero}
Suppose that $\mu \ge 0$, $0 < \nu \le 1$ and $a > 0$. Then, for $t > 0$, there holds
$$
R_{\mu,\nu}(0+,t) = \lim_{a \rightarrow 0^+} R_{\mu,\nu}(a,t) =
\begin{cases}
\delta_\mu(t) & \text{if $\mu > 0$}, \\
\delta(t) & \text{if $\mu = 0$},
\end{cases}
$$
where $\delta_\mu(t)$ is given in \eqref{delta-fun} and $\delta(t)$ is the Dirac delta function.
\end{prop}
\begin{proof}
If $\mu > 0$, then from \eqref{R-def} we get
$$
R_{\mu,\nu}(0+,t) = \lim_{a \rightarrow 0^+} R_{\mu,\nu}(a,t) = \L^{-1}\{s^{-\mu};t\} = \frac{t^{\mu - 1}}{\Gamma(\mu)} = \delta_\mu(t).
$$
Similarly, if $\mu = 0$, then
$$
R_{0,\nu}(0+,t) = \lim_{a \rightarrow 0^+} R_{0,\nu}(a,t) = \L^{-1}\{1;t\} = \delta(t).
$$
\end{proof}

The next proposition will be used when taking the spatial derivative of the solution of an associated IBVP. Note the assumption~$\mu \ge \nu$ here.
\begin{prop}
\label{R-partial-a}
If $0 < \nu \le 1$, $\mu \ge \nu$ and $a > 0$, then
\begin{equation*}
\frac{\partial R_{\mu,\nu}}{\partial a}(a,t) = -R_{\mu - \nu,\nu}(a,t).
\end{equation*}
\end{prop}
\begin{proof}
It is straightforward to see from \eqref{R-def} that
$$
\frac{\partial R_{\mu,\nu}}{\partial a}(a,t) = \L^{-1}\{s^{-\mu} \e^{-a s^\nu} (-s^\nu);t\} = -\L^{-1}\{s^{-(\mu - \nu)} \e^{-a s^\nu};t\} = -R_{\mu - \nu,\nu}(a,t).
$$
\end{proof}

The next task is to obtain a series representation for $R_{\mu,\nu}(a,t)$. Recall the Mainardi function~$M(z;\nu)$ with the series representation~\citep{Ma96}
$$
M(z;\nu) = \sum_{j = 0}^\infty \frac{(-z)^j}{j! \Gamma(-\nu j + (1 - \nu))},
$$
where $0 < \nu < 1$. It turns out to be a special case of the Wright function~$W(z;\alpha,\beta)$ with the series representation~\citep{MaPa03}
\begin{equation}
\label{W-series}
W(z;\alpha,\beta) = \sum_{j = 0}^\infty \frac{z^j}{j! \Gamma(\alpha j + \beta)}, 
\end{equation}
where $\alpha > -1$ and $\beta > 0$ (in fact, it is also valid for $\beta \in \C$). More precisely, $M(z;\nu) = W(-z;-\nu,1 - \nu)$. An interesting relation pointed out in \citet{Ro22b}, valid when $0 < \nu \le \frac{1}{2}$,  is
\begin{equation}
\label{MWR-rel}
M(a t^{-\nu};\nu) = W(-a t^{-\nu};-\nu,1 - \nu) = t^\nu R_{1 - \nu,\nu}(a,t) = t^\nu \D{}{0}{t}{-(1 - \nu)} R_{0,\nu}(a,t).
\end{equation}
We will use \eqref{MWR-rel} to derive a series representation for $R_{\mu,\nu}(a,t)$ when $\mu \ge 0$, $0 < \nu \le \frac{1}{2}$ and $a > 0$.

\begin{prop}
\label{R-series}
Let $\mu \ge 0$, $0 < \nu \le \frac{1}{2}$ and $a > 0$. A series representation for $R_{\mu,\nu}(a,t)$ is given by
$$
R_{\mu,\nu}(a,t) = t^{\mu - 1} W(- a t^{-\nu};-\nu,\mu) = \sum_{j = 0}^\infty \frac{(- a t^{-\nu})^j}{j! \Gamma(-\nu j + \mu)}.
$$
\end{prop}
\begin{proof}
The series representation~\eqref{W-series} yields
$$
W(-a t^{-\nu};-\nu,1 - \nu) = \sum_{j = 0}^\infty \frac{(-a t^{-\nu})^j}{j! \Gamma(-\nu j - \nu + 1)},
$$
which in turn gives
$$
t^{-\nu} W(-a t^{-\nu};-\nu,1 - \nu) = \sum_{j = 0}^\infty \frac{(-a)^j t^{-\nu j - \nu}}{j! \Gamma(-\nu j - \nu + 1)}.
$$
Since 
$$
R_{0,\nu}(a,t) = \D{}{0}{t}{(1 - \nu)} (t^{-\nu} W(-a t^{-\nu};-\nu,1 - \nu))
$$
from \eqref{MWR-rel}, we obtain
\begin{align*}
\D{}{0}{t}{(1 - \nu)} (t^{-\nu} W(-a t^{-\nu};-\nu,1 - \nu)) & = \D{}{0}{t}{\ceil{1 - \nu}} \D{}{0}{t}{-(\ceil{1 - \nu} - (1 - \nu))} (t^{-\nu} W(-a t^{-\nu};-\nu,1 - \nu)) \\
& = D^1 \D{}{0}{t}{-\nu} (t^{-\nu} W(-a t^{-\nu};-\nu,1 - \nu))
\end{align*}
and
\begin{align*}
\D{}{0}{t}{-\nu} (t^{-\nu} W(-a t^{-\nu};-\nu,1 - \nu)) & = \sum_{j = 0}^\infty \frac{(-a)^j}{j! \Gamma(-\nu j - \nu + 1)} \D{}{0}{t}{-\nu}(t^{-\nu j - \nu}) = \sum_{j = 0}^\infty \frac{(-a)^j t^{-\nu j}}{j! \Gamma(1 - \nu j)}.
\end{align*}
Hence
$$
R_{0,\nu}(a,t) = \D{}{0}{t}{(1 - \nu)} (t^{-\nu} W(-a t^{-\nu};-\nu,1 - \nu)) = \sum_{j = 0}^\infty \frac{(-a)^j t^{-\nu j - 1}}{j! \Gamma(-\nu j)}.
$$
Eq.~\eqref{R-basic} implies that
\begin{align*}
R_{\mu,\nu}(a,t) & = \D{}{0}{t}{-\mu} R_{0,\nu}(a,t) = \sum_{j = 0}^\infty \frac{(-a)^j}{j! \Gamma(-\nu j)} \D{}{0}{t}{-\mu} (t^{-\nu j - 1}) \\
& = t^{\mu - 1} \sum_{j = 0}^\infty \frac{(- a t^{-\nu})^j}{j! \Gamma(-\nu j + \mu)} = t^{\mu - 1} W(- a t^{-\nu};-\nu,\mu).
\end{align*}
\end{proof}

\begin{rem}
The result of Proposition~\ref{R-series} relies on the relation~\eqref{MWR-rel}, which is true if $0 < \nu \le \frac{1}{2}$. It is an open problem to determine whether the series representation is also valid for $\frac{1}{2} < \nu \le 1$.
\end{rem}

\begin{rem}
Aside from the auxiliary function~$M(z;\nu)$, \citet{MaPa03} also introduced the auxiliary function
$$
F(z;\nu) = \sum_{j = 0}^\infty \frac{(-z)^j}{j! \Gamma(-\nu j)}.
$$
It follows from Proposition~\ref{R-series} that $M(z;\nu)$ and $F(z;\nu)$ can be expressed in terms of $R_{\mu,\nu}(a,t)$ as
$$
M(a t^{-\nu};\nu) = t^\nu R_{1 - \nu,\nu}(a,t) = t^\nu \D{}{0}{t}{-(1 - \nu)} R_{0,\nu}(a,t), \quad F(a t^{-\nu};\nu) = t R_{0,\nu}(a,t),
$$
respectively. Thus we deduce another relation between $M(z;\nu)$ and $F(z;\nu)$, namely
$$
M(a t^{-\nu};\nu) = t^\nu \D{}{0}{t}{-(1 - \nu)}(t^{-1} F(a t^{-\nu};\nu)).
$$
\end{rem}

\begin{ex}
Some special values of the Wright function are known~\citep{MaPa03}:
\begin{equation}
\label{W-special}
W\Big(-z;-\frac{1}{2},\frac{1}{2}\Big) = \frac{\e^{-\frac{z^2}{4}}}{\sqrt{\pi}}, \quad W\Big(-z;-\frac{1}{2},1\Big) = 1 - \erf\Big(\frac{z}{2}\Big) = \erfc\Big(\frac{z}{2}\Big).
\end{equation}
Using Proposition~\ref{R-series}, it is not difficult to see that the second and third relations in \eqref{R-special-cases} are recovered.
\end{ex}

\begin{prop}
\label{R-integral}
If $\mu \ge 0$ and $0 < \nu \le 1$, then
$$
\int_{-\infty}^\infty \frac{1}{2} R_{\mu,\nu}(\vert z \vert,t) \, \d z = \delta_{\mu + \nu}(t),
$$
where $\delta_{\mu + \nu}(t)$ is given by \eqref{delta-fun}.
\end{prop}
\begin{proof}
The definition in \eqref{R-def} leads to
\begin{align*}
\int_{-\infty}^\infty R_{\mu,\nu}(\vert z \vert,t) \, \d z & = \int_{-\infty}^\infty \L^{-1}\{s^{-\mu} \e^{-\vert z \vert s^\nu};t\} \, \d z = \L^{-1}\Big\{\int_{-\infty}^\infty s^{-\mu} \e^{-\vert z \vert s^\nu} \, \d z;t\Big\} \\
& = 2 \L^{-1}\{s^{-(\mu + \nu)};t\} = \frac{2 t^{\mu + \nu - 1}}{\Gamma(\mu + \nu)} = 2 \delta_{\mu + \nu}(t).
\end{align*}
Note that $\int_{-\infty}^\infty \frac{1}{2} R_{\mu,\nu}(\vert z \vert,t) \, \d z = 1$ only if $\mu + \nu = 1$. This observation is related to the generation of probability distributions from the time-fractional diffusion equation discussed in \cite{Ro22b}.
\end{proof}

\section{Solution of a general IBVP for the time-fractional diffusion equation using the embedding approach}

In this section, we formulate a general IBVP for the time-fractional diffusion equation (i.e.~$0 < \nu \le \frac{1}{2}$) defined on bounded or unbounded spatial domains, and derive the analytical solution using the embedding approach.

Let $f(x)$, $g^\pm(t)$ and $\eta^\pm(t)$ be given suitable functions. Suppose that $-\infty \le \eta^-(t) < \eta^+(t) \le \infty$ for $t > 0$, which ensures that both bounded and unbounded spatial domains are taken into account. Let $a$, $b$, $c$ and $d$ be constants such that $\vert a \vert + \vert b \vert > 0$ and $\vert c \vert + \vert d \vert > 0$.

Consider the IBVP
\begin{equation}
\label{gen-IBVP}
\left\{
\begin{split}
& \D{}{}{}{2 \nu} u = \kappa \frac{\partial^2 u}{\partial x^2}, \quad \eta^-(t) < x < \eta^+(t), \quad t > 0, \\
& \Phi u(x,0+) = f(x), \quad \eta^-(0) \le x \le \eta^+(0), \\
& a u(\eta^-(t),t) + b \frac{\partial u}{\partial x}(\eta^-(t),t) = g^-(t), \quad t > 0, \\
& c u(\eta^+(t),t) + d \frac{\partial u}{\partial x}(\eta^+(t),t) = g^+(t), \quad t > 0,
\end{split}
\right.
\end{equation}
where $\D{}{}{}{2 \nu}$ is either a Caputo fractional derivative ($\D{}{}{}{2 \nu} = \D{C}{0}{t}{2 \nu}$) or a Riemann-Liouville fractional derivative ($\D{}{}{}{2 \nu} = \D{}{0}{t}{2 \nu}$). The operator~$\Phi$ defines the initial condition~(IC) through
$$
\Phi u =
\begin{cases}
u & \text{if $\D{}{}{}{2 \nu} = \D{C}{0}{t}{2 \nu}$}, \\
\D{}{0}{t}{-(1 - 2 \nu)} & \text{if $\D{}{}{}{2 \nu} = \D{}{0}{t}{2 \nu}$}.
\end{cases}
$$
The motivation behind the choice of the IC was given in \citet{Ro22b} as a natural consequence of the Laplace transform properties of the Caputo and Riemann-Liouville fractional derivatives. We assume that the IBVP~\eqref{gen-IBVP} is well posed.

\begin{rem}
In the special case when $\nu = \frac{1}{2}$, the time-fractional diffusion equation reduces to the classical diffusion equation, and the analytical solution of \eqref{gen-IBVP} was obtained in \citet{RoTh21} via the embedding method. The numerical solution of a generalisation of \eqref{gen-IBVP} with advection and reaction terms was addressed in \citet{RoTh22}.
\end{rem}

\begin{rem}
The embedding method was used in \citet{Ro22b} to provide a unified way to solve IVPs and IBVPs. However, the IBVP studied there is a very special case of \eqref{gen-IBVP}, i.e.~$\eta^-(t) = 0$, $\eta^+(t) = \infty$ and only a Dirichlet-type BC of the form~$u(x,0+) = h(t)$ for a given function~$h(t)$ was considered at the left endpoint.
\end{rem}

Let $f_\mathrm{ext}(x)$ be an extension of $f(x)$ such that $f_\mathrm{ext}(x) \vert_{\eta^-(0) \le x \le \eta^+(0)} = f(x)$. Denote by $\chi_A(x)$ the indicator function of the set~$A$, i.e.~$\chi_A(x) = 1$ if $x \in A$ and $\chi_A(x) = 0$ if $x \notin A$. We can embed the PDE and IC in \eqref{gen-IBVP} into the IVP on the real line for $v(x,t)$, namely
\begin{equation}
\label{v-IVP}
\begin{split}
& {}_{}^{}D_{}^{2 \nu} v = \kappa \frac{\partial^2 v}{\partial x^2} + F(x,t), \quad x \in \R, \quad t > 0, \\
& v(x,0) = f_\mathrm{ext}(x), \quad x \in \R, 
\end{split}
\end{equation}
where
$$
F(x,t) = \varphi^-(t) \chi_{(-\infty,\eta^-(t)]}(x) + \varphi^+(t) \chi_{[\eta^+(t),\infty)}(x) = 
\begin{cases}
\varphi^-(t) & \text{if $x \le \eta^-(t)$}, \\
0 & \text{if $\eta^-(t) < x < \eta^+(t)$}, \\
\varphi^+(t) & \text{if $x \ge \eta^+(t)$}.
\end{cases}
$$
The arbitrary functions~$\varphi^\pm(t)$ are to be determined such that the BCs in \eqref{gen-IBVP} are satisfied when we restrict $\eta^-(t) \le x \le \eta^+(t)$.

\begin{rem}
Before we proceed to give the solution of \eqref{v-IVP}, we make a few observations. We can write
\begin{align*}
& \int_0^t \int_{-\infty}^\infty \frac{1}{2 \sqrt{\kappa}} R_{\nu,\nu} \Big(\frac{\vert x - \xi \vert}{\sqrt{\kappa}},t - \tau\Big) F(\xi,\tau) \, \d \xi \, \d \tau \\
& \qquad = \int_0^t \varphi^-(\tau) \int_{-\infty}^{\eta^-(\tau)} \frac{1}{2 \sqrt{\kappa}} R_{\nu,\nu} \Big(\frac{\vert x - \xi \vert}{\sqrt{\kappa}},t - \tau\Big) \, \d \xi \, \d \tau \\
& \qquad \quad {} + \int_0^t \varphi^+(\tau) \int_{\eta^+(\tau)}^\infty \frac{1}{2 \sqrt{\kappa}} R_{\nu,\nu} \Big(\frac{\vert x - \xi \vert}{\sqrt{\kappa}},t - \tau\Big) \, \d \xi \, \d \tau.
\end{align*}
Suppose that $\eta^-(t) \le x \le \eta^+(t)$. The argument when $x = \eta^\pm(t)$ can be justified with Proposition~\ref{R-a-zero}. In the first integral on the right-hand side, noting that $-\infty < \xi \le \eta^-(\tau) \le x$, we have from \eqref{R-int-prop-2} that
\begin{align*}
\int_{-\infty}^{\eta^-(\tau)} \frac{1}{2 \sqrt{\kappa}} R_{\nu,\nu} \Big(\frac{\vert x - \xi \vert}{\sqrt{\kappa}},t - \tau\Big) \, \d \xi \, \d \tau & = \int_{-\infty}^{\eta^-(\tau)} \frac{1}{2 \sqrt{\kappa}} R_{\nu,\nu} \Big(\frac{x - \xi}{\sqrt{\kappa}},t - \tau\Big) \, \d \xi \\
& = \int_{\frac{x - \eta^-(\tau)}{\sqrt{\kappa}}}^\infty \frac{1}{2} R_{\nu,\nu}(z,t - \tau) \, \d z \\
& = \frac{1}{2} R_{2 \nu,\nu}\Big(\frac{x - \eta^-(\tau)}{\sqrt{\kappa}},t - \tau\Big).
\end{align*}
Similarly, $x \le \eta^+(\tau) \le \xi < \infty$ in the second integral, giving
\begin{align*}
\int_{\eta^+(\tau)}^\infty \frac{1}{2 \sqrt{\kappa}} R_{\nu,\nu} \Big(\frac{\vert x - \xi \vert}{\sqrt{\kappa}},t - \tau\Big) \, \d \xi & = \int_{\eta^+(\tau)}^\infty \frac{1}{2 \sqrt{\kappa}} R_{\nu,\nu} \Big(\frac{\xi - x}{\sqrt{\kappa}},t - \tau\Big) \, \d \xi \\
& = \int_{\frac{\eta^+(\tau) - x}{\sqrt{\kappa}}}^\infty \frac{1}{2} R_{\nu,\nu}(z,t - \tau) \, \d z \\
& = \frac{1}{2} R_{2 \nu,\nu}\Big(\frac{\eta^+(\tau) - x}{\sqrt{\kappa}},t - \tau\Big).
\end{align*}
Therefore 
\begin{equation}
\label{F-integral}
\begin{split}
\int_0^t \int_{-\infty}^\infty \frac{1}{2 \sqrt{\kappa}} R_{\nu,\nu} \Big(\frac{\vert x - \xi \vert}{\sqrt{\kappa}},t - \tau\Big) F(\xi,\tau) \, \d \xi \, \d \tau & = \int_0^t \frac{1}{2} R_{2 \nu,\nu}\Big(\frac{x - \eta^-(\tau)}{\sqrt{\kappa}},t - \tau\Big) \varphi^-(\tau) \, \d \tau \\
& \quad {} + \int_0^t \frac{1}{2} R_{2 \nu,\nu}\Big(\frac{\eta^+(\tau) - x}{\sqrt{\kappa}},t - \tau\Big) \varphi^+(\tau)  \, \d \tau.
\end{split}
\end{equation}
\end{rem}

We will separate the analysis of \eqref{v-IVP} according to the type of fractional derivative operator~$\D{}{}{}{2 \nu}$ being considered.

\subsection{Caputo time-fractional diffusion equation}

Suppose that $\D{}{}{}{2 \nu} = \D{C}{0}{t}{2 \nu}$. It was shown in \citet{Ro22b} that the solution of the IVP~\eqref{v-IVP} is
\begin{equation*}
\begin{split}
v(x,t) & = \int_{-\infty}^\infty \frac{1}{2 \sqrt{\kappa}} R_{1 - \nu,\nu}\Big(\frac{\vert x - \xi \vert}{\sqrt{\kappa}},t\Big) f_\mathrm{ext}(\xi) \, \d \xi \\
& \quad {} + \int_0^t \int_{-\infty}^\infty \frac{1}{2 \sqrt{\kappa}} R_{\nu,\nu} \Big(\frac{\vert x - \xi \vert}{\sqrt{\kappa}},t - \tau\Big) F(\xi,\tau) \, \d \xi \, \d \tau.
\end{split}
\end{equation*}
Hence, restricting $\eta^-(t) \le x \le \eta^+(t)$ and recalling \eqref{F-integral}, the function
\begin{equation}
\label{u-sol-1}
\begin{split}
u(x,t) & = \int_{-\infty}^\infty \frac{1}{2 \sqrt{\kappa}} R_{1 - \nu,\nu}\Big(\frac{\vert x - \xi \vert}{\sqrt{\kappa}},t\Big) f_\mathrm{ext}(\xi) \, \d \xi + \int_0^t \frac{1}{2} R_{2 \nu,\nu}\Big(\frac{x - \eta^-(\tau)}{\sqrt{\kappa}},t - \tau\Big) \varphi^-(\tau)  \, \d \tau \\
& \quad {} + \int_0^t \frac{1}{2} R_{2 \nu,\nu}\Big(\frac{\eta^+(\tau) - x}{\sqrt{\kappa}},t - \tau\Big) \varphi^+(\tau)  \, \d \tau 
\end{split}
\end{equation}
satisfies the PDE and IC of \eqref{gen-IBVP}, but not necessarily the BCs.

To verify the BCs, we need to take the partial derivative of \eqref{u-sol-1} with respect to $x$. Breaking up the first integral on the right-hand side,
\begin{align*}
u(x,t) & =  \int_{-\infty}^{x} \frac{1}{2 \sqrt{\kappa}} R_{1 - \nu,\nu}\Big(\frac{x - \xi}{\sqrt{\kappa}},t\Big) f_\mathrm{ext}(\xi) \, \d \xi - \int_\infty^x \frac{1}{2 \sqrt{\kappa}} R_{1 - \nu,\nu}\Big(\frac{\xi - x}{\sqrt{\kappa}},t\Big) f_\mathrm{ext}(\xi) \, \d \xi \\
& \quad {} + \int_0^t \frac{1}{2} R_{2 \nu,\nu}\Big(\frac{x - \eta^-(\tau)}{\sqrt{\kappa}},t - \tau\Big) \varphi^-(\tau) \, \d \tau + \int_0^t \frac{1}{2} R_{2 \nu,\nu}\Big(\frac{\eta^+(\tau) - x}{\sqrt{\kappa}},t - \tau\Big) \varphi^+(\tau)  \, \d \tau.
\end{align*}
Performing straightforward calculations with the help of Proposition~\ref{R-partial-a}, we obtain
\begin{align*}
\frac{\partial}{\partial x}\int_{-\infty}^{x} \frac{1}{2 \sqrt{\kappa}} R_{1 - \nu,\nu}\Big(\frac{x - \xi}{\sqrt{\kappa}},t\Big) f_\mathrm{ext}(\xi) \, \d \xi & = \frac{1}{2 \sqrt{\kappa}} R_{1 - \nu,\nu}(0+,t) f_\mathrm{ext}(x) \\
& \quad {} - \int_{-\infty}^{x} \frac{1}{2 \kappa} R_{1 - 2 \nu,\nu}\Big(\frac{x - \xi}{\sqrt{\kappa}},t\Big) f_\mathrm{ext}(\xi) \, \d \xi,
\end{align*}
\begin{align*}
-\frac{\partial}{\partial x}\int_\infty^x \frac{1}{2 \sqrt{\kappa}} R_{1 - \nu,\nu}\Big(\frac{\xi - x}{\sqrt{\kappa}},t\Big) f_\mathrm{ext}(\xi) \, \d \xi 
& = -\frac{1}{2 \sqrt{\kappa}} R_{1 - \nu,\nu}(0+,t) f_\mathrm{ext}(x) \\
& \quad {} + \int_x^{\infty} \frac{1}{2 \kappa} R_{1 - 2 \nu,\nu}\Big(\frac{\xi - x}{\sqrt{\kappa}},t\Big) f_\mathrm{ext}(\xi) \, \d \xi,
\end{align*}
\begin{align*}
\frac{\partial}{\partial x}\int_0^t \frac{1}{2} R_{2 \nu,\nu}\Big(\frac{x - \eta^-(\tau)}{\sqrt{\kappa}},t - \tau\Big) \varphi^-(\tau) \, \d \tau & = - \int_0^t \frac{1}{2 \sqrt{\kappa}} R_{\nu,\nu}\Big(\frac{x - \eta^-(\tau)}{\sqrt{\kappa}},t - \tau\Big) \varphi^-(\tau)  \, \d \tau
\end{align*}
and
\begin{align*}
\frac{\partial}{\partial x}\int_0^t \frac{1}{2} R_{2 \nu,\nu}\Big(\frac{\eta^+(\tau) - x}{\sqrt{\kappa}},t - \tau\Big) \varphi^+(\tau)  \, \d \tau & = \int_0^t \frac{1}{2 \sqrt{\kappa}} R_{\nu,\nu}\Big(\frac{\eta^+(\tau) - x}{\sqrt{\kappa}},t - \tau\Big) \varphi^+(\tau)  \, \d \tau.
\end{align*}
Combining these integrals, we get
\begin{equation}
\label{u-sol-1-der}
\begin{split}
\frac{\partial u}{\partial x}(x,t) 
& = - \int_{-\infty}^{x} \frac{1}{2 \kappa} R_{1 - 2 \nu,\nu}\Big(\frac{x - \xi}{\sqrt{\kappa}},t\Big) f_\mathrm{ext}(\xi) \, \d \xi + \int_x^{\infty} \frac{1}{2 \kappa} R_{1 - 2 \nu,\nu}\Big(\frac{\xi - x}{\sqrt{\kappa}},t\Big) f_\mathrm{ext}(\xi) \, \d \xi \\ 
& \quad {} - \int_0^t \frac{1}{2 \sqrt{\kappa}} R_{\nu,\nu}\Big(\frac{x - \eta^-(\tau)}{\sqrt{\kappa}},t - \tau\Big) \varphi^-(\tau)  \, \d \tau \\
& \quad {} + \int_0^t \frac{1}{2 \sqrt{\kappa}} R_{\nu,\nu}\Big(\frac{\eta^+(\tau) - x}{\sqrt{\kappa}},t - \tau\Big) \varphi^+(\tau)  \, \d \tau.
\end{split}
\end{equation}

We introduce some simplifying notation. Identify $\eta^-_1$ with $\eta^-(t)$, $\eta^-_2$ with $\eta^-(\tau)$, $\eta^+_1$ with $\eta^+(t)$ and $\eta^+_2$ with $\eta^+(\tau)$. Define the kernel functions
\begin{align*}
K_{11}(\eta^-_1,\eta^-_2,\eta^+_1,\eta^+_2,t) & = \frac{a}{2}  R_{2 \nu,\nu}\Big(\frac{\eta^-_1 - \eta^-_2}{\sqrt{\kappa}},t\Big) - \frac{b}{2 \sqrt{\kappa}} R_{\nu,\nu}\Big(\frac{\eta^-_1 - \eta^-_2}{\sqrt{\kappa}},t\Big), \\
K_{12}(\eta^-_1,\eta^-_2,\eta^+_1,\eta^+_2,t) & = \frac{a}{2} R_{2 \nu,\nu}\Big(\frac{\eta^+_2 - \eta^-_1}{\sqrt{\kappa}},t\Big) + \frac{b}{2 \sqrt{\kappa}} R_{\nu,\nu}\Big(\frac{\eta^+_2 - \eta^-_1}{\sqrt{\kappa}},t\Big), \\
K_{21}(\eta^-_1,\eta^-_2,\eta^+_1,\eta^+_2,t) & = \frac{c}{2} R_{2 \nu,\nu}\Big(\frac{\eta^+_1 - \eta^-_2}{\sqrt{\kappa}},t\Big) - \frac{d}{2 \sqrt{\kappa}} R_{\nu,\nu}\Big(\frac{\eta^+_1 - \eta^-_2}{\sqrt{\kappa}},t\Big), \\
K_{22}(\eta^-_1,\eta^-_2,\eta^+_1,\eta^+_2,t) & = \frac{c}{2} R_{2 \nu,\nu}\Big(\frac{\eta^+_2 - \eta^+_1}{\sqrt{\kappa}},t\Big) + \frac{d}{2 \sqrt{\kappa}} R_{\nu,\nu}\Big(\frac{\eta^+_2 - \eta^+_1}{\sqrt{\kappa}},t\Big).
\end{align*}
Moreover, define
\begin{equation}
\label{h-minus-1}
\begin{split}
h^-(t) & = g^-(t) - \int_{-\infty}^\infty \frac{a}{2 \sqrt{\kappa}} R_{1 - \nu,\nu}\Big(\frac{\vert \eta^-(t) - \xi \vert}{\sqrt{\kappa}},t\Big) f_\mathrm{ext}(\xi) \, \d \xi \\
& \quad {} +\int_{-\infty}^{\eta^-(t)} \frac{b}{2 \kappa} R_{1 - 2 \nu,\nu}\Big(\frac{\eta^-(t) - \xi}{\sqrt{\kappa}},t\Big) f_\mathrm{ext}(\xi) \, \d \xi \\
& \quad {} - \int_{\eta^-(t)}^{\infty} \frac{b}{2 \kappa} R_{1 - 2 \nu,\nu}\Big(\frac{\xi - \eta^-(t)}{\sqrt{\kappa}},t\Big) f_\mathrm{ext}(\xi) \, \d \xi
\end{split}
\end{equation}
and
\begin{equation}
\label{h-plus-1}
\begin{split}
h^+(t) & = g^+(t) - \int_{-\infty}^\infty \frac{c}{2 \sqrt{\kappa}} R_{1 - \nu,\nu}\Big(\frac{\vert \eta^+(t) - \xi \vert}{\sqrt{\kappa}},t\Big) f_\mathrm{ext}(\xi) \, \d \xi \\
& \quad {} + \int_{-\infty}^{\eta^+(t)} \frac{d}{2 \kappa} R_{1 - 2 \nu,\nu}\Big(\frac{\eta^+(t) - \xi}{\sqrt{\kappa}},t\Big) f_\mathrm{ext}(\xi) \, \d \xi \\
& \quad {} - \int_{\eta^+(t)}^{\infty} \frac{d}{2 \kappa} R_{1 - 2 \nu,\nu}\Big(\frac{\xi - \eta^+(t)}{\sqrt{\kappa}},t\Big) f_\mathrm{ext}(\xi) \, \d \xi.
\end{split}
\end{equation}
Substituting the above expressions into the BCs in \eqref{gen-IBVP}, the left BC becomes
\begin{equation}
\label{left-BC}
\begin{split}
& \int_0^t K_{11}(\eta^-(t),\eta^-(\tau),\eta^+(t),\eta^+(\tau),t - \tau) \varphi^-(\tau) \, \d \tau \\
& \quad {} + \int_0^t K_{12}(\eta^-(t),\eta^-(\tau),\eta^+(t),\eta^+(\tau),t - \tau) \varphi^+(\tau) \, \d \tau = h^-(t),
\end{split}
\end{equation}
while the right BC simplifies to
\begin{equation}
\label{right-BC}
\begin{split}
& \int_0^t K_{21}(\eta^-(t),\eta^-(\tau),\eta^+(t),\eta^+(\tau),t - \tau) \varphi^-(\tau) \, \d \tau \\
& \quad {} + \int_0^t K_{22}(\eta^-(t),\eta^-(\tau),\eta^+(t),\eta^+(\tau),t - \tau) \varphi^+(\tau) \, \d \tau = h^+(t).
\end{split}
\end{equation}

In summary, the analytical solution of the IBVP~\eqref{gen-IBVP} for the Caputo time-fractional diffusion equation is \eqref{u-sol-1}, where $\varphi^\pm(t)$ satisfy the pair of linear Volterra integral equations of the first kind described by \eqref{left-BC} and \eqref{right-BC}. The functions~$h^\pm(t)$ are given in \eqref{h-minus-1} and \eqref{h-plus-1}. Note that other choices of defining $f_\mathrm{ext}(x)$ will result in a corresponding adjustment of $h^\pm(t)$, yielding the same solution in the end.

\subsection{Riemann-Liouville time-fractional diffusion equation}

Now take $\D{}{}{}{2 \nu} = \D{}{0}{t}{2 \nu}$. As the calculations are similar to the Caputo case, we just give the final result. The analytical solution of the IBVP~\eqref{gen-IBVP} for the Riemann-Liouville time-fractional diffusion equation is
\begin{equation}
\label{u-sol-2}
\begin{split}
u(x,t) & = \int_{-\infty}^\infty \frac{1}{2 \sqrt{\kappa}} R_{\nu,\nu}\Big(\frac{\vert x - \xi \vert}{\sqrt{\kappa}},t\Big) f_\mathrm{ext}(\xi) \, \d \xi + \int_0^t \frac{1}{2} R_{2 \nu,\nu}\Big(\frac{x - \eta^-(\tau)}{\sqrt{\kappa}},t - \tau\Big) \varphi^-(\tau)  \, \d \tau \\
& \quad {} + \int_0^t \frac{1}{2} R_{2 \nu,\nu}\Big(\frac{\eta^+(\tau) - x}{\sqrt{\kappa}},t - \tau\Big) \varphi^+(\tau)  \, \d \tau.
\end{split}
\end{equation}
Note that one difference between \eqref{u-sol-2} and \eqref{u-sol-1} is in the first integral on the right-hand side. The functions~$\varphi^\pm(t)$ satisfy the pair of linear Volterra integral equations of the first kind also described by \eqref{left-BC} and \eqref{right-BC} but $h^\pm(t)$ are given by
\begin{equation}
\label{h-minus-2}
\begin{split}
h^-(t) & = g^-(t) - \int_{-\infty}^\infty \frac{a}{2 \sqrt{\kappa}} R_{\nu,\nu}\Big(\frac{\vert \eta^-(t) - \xi \vert}{\sqrt{\kappa}},t\Big) f_\mathrm{ext}(\xi) \, \d \xi \\
& \quad {} + \int_{-\infty}^{\eta^-(t)} \frac{b}{2 \kappa} R_{0,\nu}\Big(\frac{\eta^-(t) - \xi}{\sqrt{\kappa}},t\Big) f_\mathrm{ext}(\xi) \, \d \xi - \int_{\eta^-(t)}^{\infty} \frac{b}{2 \kappa} R_{0,\nu}\Big(\frac{\xi - \eta^-(t)}{\sqrt{\kappa}},t\Big) f_\mathrm{ext}(\xi) \, \d \xi
\end{split}
\end{equation}
and
\begin{equation}
\label{h-plus-2}
\begin{split}
h^+(t) & = g^+(t) - \int_{-\infty}^\infty \frac{c}{2 \sqrt{\kappa}} R_{\nu,\nu}\Big(\frac{\vert \eta^+(t) - \xi \vert}{\sqrt{\kappa}},t\Big) f_\mathrm{ext}(\xi) \, \d \xi \\
& \quad {} + \int_{-\infty}^{\eta^+(t)} \frac{d}{2 \kappa} R_{0,\nu}\Big(\frac{\eta^+(t) - \xi}{\sqrt{\kappa}},t\Big) f_\mathrm{ext}(\xi) \, \d \xi - \int_{\eta^+(t)}^{\infty} \frac{d}{2 \kappa} R_{0,\nu}\Big(\frac{\xi - \eta^+(t)}{\sqrt{\kappa}},t\Big) f_\mathrm{ext}(\xi) \, \d \xi.
\end{split}
\end{equation}

For later use, we note that
\begin{equation}
\label{u-sol-2-der}
\begin{split}
\frac{\partial u}{\partial x}(x,t) & = - \int_{-\infty}^{x} \frac{1}{2 \kappa} R_{0,\nu}\Big(\frac{x - \xi}{\sqrt{\kappa}},t\Big) f_\mathrm{ext}(\xi) \, \d \xi + \int_x^{\infty} \frac{1}{2 \kappa} R_{0,\nu}\Big(\frac{\xi - x}{\sqrt{\kappa}},t\Big) f_\mathrm{ext}(\xi) \, \d \xi \\ 
& \quad {} - \int_0^t \frac{1}{2 \sqrt{\kappa}} R_{\nu,\nu}\Big(\frac{x - \eta^-(\tau)}{\sqrt{\kappa}},t - \tau\Big) \varphi^-(\tau)  \, \d \tau \\
& \quad {} + \int_0^t \frac{1}{2 \sqrt{\kappa}} R_{\nu,\nu}\Big(\frac{\eta^+(\tau) - x}{\sqrt{\kappa}},t - \tau\Big) \varphi^+(\tau)  \, \d \tau.
\end{split}
\end{equation}

\begin{rem}
As to be expected, when $\nu = \frac{1}{2}$, the Caputo solution~\eqref{u-sol-1} and the Riemann-Liouville solution~\eqref{u-sol-2} become identical and recover the analytical solution for the corresponding IBVP for the classical diffusion equation obtained in \citet{RoTh21}. 
\end{rem}

\section{Solutions of moving boundary problems associated with the time-fractional diffusion equation}

We are now ready to find analytical solutions of moving boundary problems for the time-fractional diffusion equation. Two representative examples will be considered with bounded and unbounded spatial domains. More general moving boundary problems can be handled in a similar fashion.

\begin{ex}
Consider the moving boundary problem
\begin{equation}
\left\{
\label{free-prob-1}
\begin{split}
& \D{}{}{}{2 \nu} u = \frac{\partial^2 u}{\partial x^2}, \quad 0 < x < \eta(t), \quad t > 0, \\
& u(x,0) = u_0 \chi_{(0,\infty)}(x), \quad 0 \le x < \infty, \\
& u(0,t) = 1, \quad u(\eta(t),t) = 0, \quad t > 0, \\
& \D{}{}{}{2 \nu} \eta(t) = -\frac{1}{r} \frac{\partial u}{\partial x}(\eta(t),t), \quad t > 0.
\end{split}
\right.
\end{equation}
Here, $r$ and $u_0$ are positive constants and $\eta(t)$ is the moving boundary. The goal is to find $u(x,t)$ and $\eta(t)$. 

When $\nu = \frac{1}{2}$, \eqref{free-prob-1} reduces to a classical Stefan problem for the melting of ice over a one-dimensional semi-infinite spatial domain~\citep{Cr84,Hi87}. In this context, the PDE under consideration is the heat equation. The interval~$[0,\eta(t)]$ is the region occupied by water. The last equation in \eqref{free-prob-1} is also known as the Stefan condition and $r$ is the ratio of latent to specific sensible heat. However, when $0 < \nu < \frac{1}{2}$, the physical interpretation of the problem in the context of melting of ice is not necessarily valid and we therefore study the IBVP~\eqref{free-prob-1} strictly from a theoretical perspective. 

Comparing \eqref{free-prob-1} with \eqref{gen-IBVP}, we identify $\eta^-(t) = 0$, $\eta^+(t) = \eta(t)$, $\kappa = 1$, $a = 1$, $b = 0$, $c = 1$, $d = 0$, $g^-(t) = 1$, $g^+(t) = 0$ and $f(x) = u_0 \chi_{(0,\infty)}(x)$. The last equation in \eqref{free-prob-1} provides a condition (`fractional Stefan condition') for the moving boundary~$\eta(t)$. Take $f_\mathrm{ext}(x) = u_0 \chi_{(-\infty,0) \cup (0,\infty)}(x)$ for all $x \in \R$ for instance.

Using Proposition~\ref{R-integral}, we deduce that
\begin{equation*}
\int_{-\infty}^\infty \frac{1}{2} R_{\mu,\nu}(\vert x - \xi\vert,t) \, \d \xi = \int_{-\infty}^\infty \frac{1}{2} R_{\mu,\nu}(\vert z \vert,t) \, \d z = \delta_{\mu + \nu}(t)
\end{equation*}
for any $x \in \R$. In particular,
\begin{equation}
\label{R-int-real}
\int_{-\infty}^\infty \frac{1}{2} R_{1 - \nu,\nu}(\vert x - \xi \vert,t) \, \d \xi = 1, \quad \int_{-\infty}^\infty \frac{1}{2} R_{\nu,\nu}(\vert x - \xi \vert,t) \, \d \xi = \delta_{2 \nu}(t).
\end{equation}
Assuming that $\varphi^+(t) = 0$ so as to be able to do some explicit calculations, \eqref{u-sol-1} and \eqref{u-sol-2} respectively give
\begin{equation*}
u(x,t) =
\begin{cases}
u_0 + \int_0^t \frac{1}{2} R_{2 \nu,\nu}(x,t - \tau) \varphi^-(\tau)  \, \d \tau & \text{if $\D{}{}{}{2 \nu} = \D{C}{0}{t}{2 \nu}$}, \\
u_0 \delta_{2 \nu}(t) + \int_0^t \frac{1}{2} R_{2 \nu,\nu}(x,t - \tau) \varphi^-(\tau)  \, \d \tau & \text{if $\D{}{}{}{2 \nu} = \D{}{0}{t}{2 \nu}$}.
\end{cases}
\end{equation*}
Eqs.~\eqref{h-minus-1}, \eqref{h-minus-2}, \eqref{h-plus-1} and \eqref{h-plus-2} yield
\begin{equation*}
h^-(t) = 
\begin{cases}
1 - u_0 & \text{if $\D{}{}{}{2 \nu} = \D{C}{0}{t}{2 \nu}$}, \\
1 - u_0 \delta_{2 \nu}(t) & \text{if $\D{}{}{}{2 \nu} = \D{}{0}{t}{2 \nu}$},
\end{cases} \quad 
h^+(t) = 
\begin{cases}
-u_0 & \text{if $\D{}{}{}{2 \nu} = \D{C}{0}{t}{2 \nu}$}, \\
-u_0 \delta_{2 \nu}(t) & \text{if $\D{}{}{}{2 \nu} = \D{}{0}{t}{2 \nu}$}.
\end{cases}
\end{equation*}

Next, let us look at the left BC. Suppose that $\D{}{}{}{2 \nu} = \D{C}{0}{t}{2 \nu}$. Eq.~\eqref{left-BC} gives
$$
\int_0^t \frac{1}{2}  R_{2 \nu,\nu}(0+,t - \tau) \varphi^-(\tau) \, \d \tau = 1 - u_0 \quad \text{or} \quad \D{}{0}{t}{-2 \nu}\varphi^-(t) = 2 (1 - u_0)
$$
using Proposition~\ref{R-a-zero} and \eqref{conv-int}. If $\Phi^-(s) = \L\{\varphi^-(t);s\}$, then 
$$
\Phi^-(s) = \frac{2 (1 - u_0)}{s^{1 - 2 \nu}}.
$$ 
Therefore
$$
\varphi^-(t) = 2 (1 - u_0) \delta_{1 - 2 \nu}(t)
$$
for the Caputo case. Now suppose that $\D{}{}{}{2 \nu} = \D{}{0}{t}{2 \nu}$. This time \eqref{left-BC} gives
$$
\int_0^t \frac{1}{2}  R_{2 \nu,\nu}(0+,t - \tau) \varphi^-(\tau) \, \d \tau = 1 - u_0 \delta_{2 \nu}(t) \quad \text{or} \quad \D{}{0}{t}{-2 \nu}\varphi^-(t) = 2 [1 - u_0 \delta_{2 \nu}(t)].
$$
Then
$$
\Phi^-(s) = \frac{2}{s^{1 - 2 \nu}} - 2 u_0,
$$
which yields
$$
\varphi^-(t) = 2 \delta_{1 - 2 \nu}(t) - 2 u_0 \delta(t)
$$
for the Riemann-Liouville case. Summarising, from the left BC~\eqref{left-BC} we deduce that
\begin{equation*}
\varphi^-(t) = 
\begin{cases}
2 (1 - u_0) \delta_{1 - 2 \nu}(t) & \text{if $\D{}{}{}{2 \nu} = \D{C}{0}{t}{2 \nu}$}, \\
2 \delta_{1 - 2 \nu}(t) - 2 u_0 \delta(t) & \text{if $\D{}{}{}{2 \nu} = \D{}{0}{t}{2 \nu}$}.
\end{cases}
\end{equation*}

We now examine the right BC starting with $\D{}{}{}{2 \nu} = \D{C}{0}{t}{2 \nu}$. From \eqref{right-BC} we see that
$$
\int_0^t \frac{1}{2} R_{2 \nu,\nu}(\eta(t),t - \tau) \varphi^-(\tau) \, \d \tau  = -u_0.
$$
But \eqref{conv-int}, \eqref{R-basic} and the semigroup property for the Riemann-Liouville fractional integral lead to
\begin{align*}
& \int_0^t \frac{1}{2} R_{2 \nu,\nu}(\eta(t),t - \tau) \varphi^-(\tau) \, \d \tau = \int_0^t (1 - u_0) R_{2 \nu,\nu}(\eta(t),t - \tau)  \delta_{1 - 2 \nu}(\tau) \, \d \tau \\
& \qquad = (1 - u_0) {}_{0}^{}D_{t}^{-(1 -2 \nu)} R_{2 \nu,\nu}(\eta(t),t) = (1 - u_0) {}_{0}^{}D_{t}^{-(1 -2 \nu)} {}_{0}^{}D_{t}^{-2 \nu} R_{0,\nu}(\eta(t),t) \\
& \qquad = (1 - u_0) {}_{0}^{}D_{t}^{-1} R_{0,\nu}(\eta(t),t) = (1 - u_0) R_{1,\nu}(\eta(t),t).
\end{align*}
Hence the right BC for the Caputo case becomes
$$
R_{1,\nu}(\eta(t),t) = -\frac{u_0}{1 - u_0}.
$$
Now let $\D{}{}{}{2 \nu} = \D{}{0}{t}{2 \nu}$. Eq.~\eqref{right-BC} in this case is
$$
\int_0^t \frac{1}{2} R_{2 \nu,\nu}(\eta(t),t - \tau) \varphi^-(\tau) \, \d \tau = -u_0 \delta_{2 \nu}(t).
$$
We have from \eqref{conv-int}, \eqref{R-basic} and the semigroup property for the Riemann-Liouville fractional integral that
\begin{align*}
& \int_0^t \frac{1}{2} R_{2 \nu,\nu}(\eta(t),t - \tau) \varphi^-(\tau) \, \d \tau = \int_0^t \frac{1}{2} R_{2 \nu,\nu}(\eta(t),t - \tau) [2 \delta_{1 - 2 \nu}(\tau) - 2 u_0 \delta(\tau)] \, \d \tau \\
& \qquad = {}_{0}^{}D_{t}^{-(1 - 2 \nu)} R_{2 \nu,\nu}(\eta(t),t) - u_0 R_{2 \nu,\nu}(\eta(t),t) \\
& \qquad = {}_{0}^{}D_{t}^{-(1 -2 \nu)} {}_{0}^{}D_{t}^{-2 \nu} R_{0,\nu}(\eta(t),t) - u_0 R_{2 \nu,\nu}(\eta(t),t) \\
& \qquad = {}_{0}^{}D_{t}^{-1} R_{0,\nu}(\eta(t),t) - u_0 R_{2 \nu,\nu}(\eta(t),t) = R_{1,\nu}(\eta(t),t) - u_0 R_{2 \nu,\nu}(\eta(t),t).
\end{align*}
Therefore the right BC for the Riemann-Liouville case becomes
$$
R_{1,\nu}(\eta(t),t) - u_0 R_{2 \nu,\nu}(\eta(t),t) = -u_0 \delta_{2 \nu}(t).
$$
In summary, the right BC~\eqref{right-BC} is equivalent to
\begin{equation}
\label{free-prob-1-right-BC}
\begin{cases}
R_{1,\nu}(\eta(t),t) = -\frac{u_0}{1 - u_0} & \text{if $\D{}{}{}{2 \nu} = \D{C}{0}{t}{2 \nu}$}, \\
R_{1,\nu}(\eta(t),t) - u_0 R_{2 \nu,\nu}(\eta(t),t) = -u_0 \delta_{2 \nu}(t) & \text{if $\D{}{}{}{2 \nu} = \D{}{0}{t}{2 \nu}$}.
\end{cases}
\end{equation}

Finally, we consider the `fractional Stefan condition'. Observe in \eqref{u-sol-1-der} and \eqref{u-sol-2-der} that
$$
-\int_{-\infty}^{x} \frac{u_0}{2} R_{\mu,\nu}(x - \xi,t) + \int_x^{\infty} \frac{u_0}{2} R_{\mu,\nu}(\xi - x,t) \, \d \xi = 0
$$
for any $\mu \ge 0$. If $\D{}{}{}{2 \nu} = \D{C}{0}{t}{2 \nu}$, then similar arguments as above give
\begin{align*}
\frac{\partial u}{\partial x}(x,t) 
& = -\int_0^t \frac{1}{2} R_{\nu,\nu}(x,t - \tau) \varphi^-(\tau)  \, \d \tau = -\int_0^t (1 - u_0) R_{\nu,\nu}(x,t - \tau) \delta_{1 - 2 \nu}(\tau)  \, \d \tau \\
& = -(1 - u_0) {}_{0}^{}D_{t}^{-(1 -2 \nu)} R_{\nu,\nu}(x,t) = -(1 - u_0) {}_{0}^{}D_{t}^{-(1 -2 \nu)} {}_{0}^{}D_{t}^{-\nu} R_{0,\nu}(x,t) \\
& = -(1 - u_0) {}_{0}^{}D_{t}^{-(1 - \nu)} R_{0,\nu}(x,t) = -(1 - u_0) R_{1 - \nu,\nu}(x,t).
\end{align*}
On the other hand, if $\D{}{}{}{2 \nu} = \D{}{0}{t}{2 \nu}$, then
\begin{align*}
\frac{\partial u}{\partial x}(x,t) 
& = -\int_0^t \frac{1}{2} R_{\nu,\nu}(x,t - \tau) \varphi^-(\tau)  \, \d \tau = - \int_0^t R_{\nu,\nu}(x,t - \tau) [\delta_{1 - 2 \nu}(\tau) - u_0 \delta(\tau)]  \, \d \tau \\
& = -{}_{0}^{}D_{t}^{-(1 -2 \nu)} R_{\nu,\nu}(x,t) + u_0 R_{\nu,\nu}(x,t) = -{}_{0}^{}D_{t}^{-(1 -2 \nu)} {}_{0}^{}D_{t}^{-\nu} R_{0,\nu}(x,t) + u_0 R_{\nu,\nu}(x,t) \\
& = -{}_{0}^{}D_{t}^{-(1 - \nu)} R_{0,\nu}(x,t) + u_0 R_{\nu,\nu}(x,t) = -R_{1 - \nu,\nu}(x,t) + u_0 R_{\nu,\nu}(x,t).
\end{align*}
Summarising, the `fractional Stefan condition' becomes
\begin{equation}
\label{free-prob-1-stefan}
-r \D{}{}{}{2 \nu} \eta(t) = 
\begin{cases}
-(1 - u_0) R_{1 - \nu,\nu}(\eta(t),t) & \text{if $\D{}{}{}{2 \nu} = \D{C}{0}{t}{2 \nu}$}, \\
-R_{1 - \nu,\nu}(\eta(t),t) + u_0 R_{\nu,\nu}(\eta(t),t) & \text{if $\D{}{}{}{2 \nu} = \D{}{0}{t}{2 \nu}$}.
\end{cases}
\end{equation}

It remains to determine $\eta(t)$. Looking at the series representation in Proposition~\ref{R-series} and the known similarity solution of the classical diffusion equation when $\nu = \frac{1}{2}$, we propose the ansatz~$\eta(t) = 2 \alpha t^\nu$ for some constant~$\alpha$ to be determined. Then
$$
\D{C}{0}{t}{2 \nu} \eta(t) = \D{}{0}{t}{2 \nu} \eta(t) = \frac{2 \alpha \Gamma(1 + \nu) t^{-\nu}}{\Gamma(1 - \nu)}
$$
and
$$
R_{\mu,\nu}(\eta(t),t) = t^{\mu - 1} W(-2 \alpha;-\nu,\mu)
$$
for any $\mu \ge 0$. If $\D{}{}{}{2 \nu} = \D{C}{0}{t}{2 \nu}$ in the right BC~\eqref{free-prob-1-right-BC}, then
\begin{equation}
\label{free-prob-1-trans-1}
W(-2 \alpha;-\nu,1) = -\frac{u_0}{1 - u_0},
\end{equation}
which is a transcendental equation involving $\alpha$ and $u_0$. However, if $\D{}{}{}{2 \nu} = \D{}{0}{t}{2 \nu}$ in the right BC~\eqref{free-prob-1-right-BC}, then
$$
W(-2 \alpha;-\nu,1)- u_0 t^{2 \nu - 1} W(-2 \alpha;-\nu,2 \nu) = -\frac{u_0 t^{2 \nu - 1}}{\Gamma(2 \nu)},
$$
which becomes an identity only when $\nu = \frac{1}{2}$. Hence we immediately conclude, without needing to verify the corresponding `fractional Stefan condition' in \eqref{free-prob-1-stefan}, that the ansatz~$\eta(t) = 2 \alpha t^\nu$ will not work when $0 < \nu < \frac{1}{2}$ for the Riemann-Liouville case. Taking $\D{}{}{}{2 \nu} = \D{C}{0}{t}{2 \nu}$ in the `fractional Stefan condition'~\eqref{free-prob-1-stefan}, we obtain
$$
-\frac{2 \alpha r \Gamma(1 + \nu) t^{-\nu}}{\Gamma(1 - \nu)} = -(1 - u_0) t^{-\nu} W(-2 \alpha;-\nu,1 - \nu) 
$$
or
\begin{equation}
\label{free-prob-1-trans-2}
\frac{2 \alpha r \Gamma(1 +  \nu)}{(1 - u_0) \Gamma(1 - \nu)} = W(-2 \alpha;-\nu,1 - \nu),
\end{equation}
another transcendental equation involving $\alpha$ and $u_0$. From \eqref{free-prob-1-trans-1} we can solve
$$
u_0 = -\frac{W(-2 \alpha;-\nu,1)}{1 - W(-2 \alpha;-\nu,1)}, \quad 1 - u_0 = \frac{1}{1 - W(-2 \alpha;-\nu,1)}.
$$
Substituting these into \eqref{free-prob-1-trans-2}, we get a transcendental equation only for $\alpha$, namely
\begin{equation}
\label{free-prob-1-trans-3}
\frac{2 \alpha r \Gamma(1 +  \nu)}{\Gamma(1 - \nu)} [1 - W(-2 \alpha;-\nu,1)] = W(-2 \alpha;-\nu,1 - \nu).
\end{equation}
Therefore 
\begin{equation*}
\begin{split}
u(x,t) & = u_0 + \int_0^t \frac{1}{2} R_{2 \nu,\nu}(x,t - \tau) 2 (1 - u_0) \delta_{1 - 2 \nu}(\tau) \, \d \tau = u_0 + (1 - u_0) \D{}{0}{t}{-(1 - 2 \nu)} R_{2 \nu,\nu}(x,t) \\
& = u_0 + (1 - u_0) \D{}{0}{t}{-(1 - 2 \nu)} \D{}{0}{t}{-2 \nu} R_{0,\nu}(x,t) = u_0 + (1 - u_0) \D{}{0}{t}{-1} R_{0,\nu}(x,t) \\
& = u_0 + (1 - u_0) R_{1,\nu}(x,t) = \frac{R_{1,\nu}(x,t) - W(-2 \alpha;-\nu,1)}{1 - W(-2 \alpha;-\nu,1)}
\end{split}
\end{equation*}
and the analytical solution of the moving boundary problem for the Caputo case is
\begin{equation}
\label{free-prob-1-sol}
u(x,t) = \frac{R_{1,\nu}(x,t) - W(-2 \alpha;-\nu,1)}{1 - W(-2 \alpha;-\nu,1)}, \quad \eta(t) = 2 \alpha t^\nu,
\end{equation}
where $\alpha$ satisfies the transcendental equation~\eqref{free-prob-1-trans-3}.

\begin{rem}
When $\nu = \frac{1}{2}$, \eqref{W-special} yields
$$
W\Big(-2 \alpha;-\frac{1}{2},\frac{1}{2}\Big) = \frac{\e^{-\alpha^2}}{\sqrt{\pi}}, \quad W\Big(-2\alpha;-\frac{1}{2},1\Big) = 1 - \erf(\alpha),
$$
while \eqref{R-special-cases} gives 
$$
R_{1,\frac{1}{2}}(x,t) = \erfc\Big(\frac{x}{2 \sqrt{t}}\Big) = 1 - \erf\Big(\frac{x}{2 \sqrt{t}}\Big).
$$ 
Therefore \eqref{free-prob-1-sol} simplifies to
$$
u(x,t) = \frac{\erfc(\frac{x}{2 \sqrt{t}}) - 1 + \erf(\alpha)}{\erf(\alpha)} = 1 - \frac{\erf(\frac{x}{2 \sqrt{t}})}{\erf(\alpha)}, \quad \eta(t) = 2 \alpha \sqrt{t},
$$
where $\alpha$ satisfies the transcendental equation
$$
r \sqrt{\pi} \alpha \erf(\alpha) \e^{\alpha^2} = 1.
$$
This is of course the well-known Neumann solution of the given Stefan problem for the heat equation~\citep{Cr84,Hi87} typically obtained through a similarity analysis.
\end{rem}
\end{ex}

\begin{ex}
Consider the moving boundary problem
\begin{equation}
\label{free-prob-2}
\left\{
\begin{split}
& \D{}{}{}{2 \nu} u = \frac{\partial^2 u}{\partial x^2}, \quad \eta(t) < x < \infty, \quad t > 0, \\
& u(x,0) = -1, \quad 0 \le x < \infty, \\
& u(\eta(t),t) = 0, \quad u(\infty,t) = -1, \quad t > 0, \\
& \D{}{}{}{2 \nu} \eta(t) = \frac{1}{r} \Big[1 + \frac{\partial u}{\partial x}(\eta(t),t)\Big], \quad t > 0,
\end{split}
\right.
\end{equation}
where $r$ is a positive constant and $\eta(t)$ is the moving boundary. Again, we wish to find $u(x,t)$ and $\eta(t)$. 

When $\nu = \frac{1}{2}$, \eqref{free-prob-2} reduces to a Stefan problem involving a single-phase, semi-infinite, subcooled material. One application is the determination of whether ice melts or water freezes when hot water is thrown over cold ice~\citep{Hu89}. The mathematical formulation for the heat equation is a Stefan problem with a constant heat source term in the condition at the boundary~\citep{KiRi00}. Furthermore, a related industrial process is ablation, i.e.~a mass is removed from an object by vapourisation or similar erosive processes~\citep{MiMy08,Mi12,MiMy12}. As in the previous example, the same physical interpretation when $0 < \nu < \frac{1}{2}$ is not necessarily valid so that our interest here is theoretical. We also refer to the last equation in \eqref{free-prob-2} as a `fractional Stefan condition'.

Comparing \eqref{free-prob-2} with \eqref{gen-IBVP}, we identify $\eta^-(t) = \eta(t)$, $\eta^+(t) = \infty$, $\kappa = 1$, $a = 1$, $b = 0$, $c = 1$, $d = 0$, $g^-(t) = 0$, $g^+(t) = -1$ and $f(x) = -1$. The last equation in \eqref{free-prob-1} provides a condition for the moving boundary~$\eta(t)$. Take $f_\mathrm{ext}(x) = -1$ for all $x \in \R$ for example.

Note that $R_{\mu,\nu}(\infty,t) = \lim_{a \rightarrow \infty} \L^{-1}\{s^{-\mu} \e^{-a s^\nu};t\} = 0$. Using \eqref{u-sol-1} and \eqref{u-sol-2}, we have
\begin{equation}
\label{free-prob-2-u}
u(x,t) =
\begin{cases}
-1 + \int_0^t \frac{1}{2} R_{2 \nu,\nu}(x - \eta(\tau),t - \tau) \varphi^-(\tau)  \, \d \tau & \text{if $\D{}{}{}{2 \nu} = \D{C}{0}{t}{2 \nu}$}, \\
-\delta_{2 \nu}(t) + \int_0^t \frac{1}{2} R_{2 \nu,\nu}(x - \eta(\tau),t - \tau) \varphi^-(\tau)  \, \d \tau & \text{if $\D{}{}{}{2 \nu} = \D{}{0}{t}{2 \nu}$}.
\end{cases}
\end{equation}
Eqs.~\eqref{h-minus-1}, \eqref{h-minus-2}, \eqref{h-plus-1} and \eqref{h-plus-2} yield
\begin{equation*}
h^-(t) = 
\begin{cases}
1 & \text{if $\D{}{}{}{2 \nu} = \D{C}{0}{t}{2 \nu}$}, \\
\delta_{2 \nu}(t) & \text{if $\D{}{}{}{2 \nu} = \D{}{0}{t}{2 \nu}$},
\end{cases} \quad 
h^+(t) = 
\begin{cases}
-1 & \text{if $\D{}{}{}{2 \nu} = \D{C}{0}{t}{2 \nu}$}, \\
-1 & \text{if $\D{}{}{}{2 \nu} = \D{}{0}{t}{2 \nu}$}.
\end{cases}
\end{equation*}

From \eqref{left-BC} we deduce that the left BC is
\begin{equation}
\label{free-prob-2-phi}
\begin{cases}
\int_0^t \frac{1}{2}  R_{2 \nu,\nu}(\eta(t) - \eta(\tau),t - \tau) \varphi^-(\tau) \, \d \tau = 1 & \text{if $\D{}{}{}{2 \nu} = \D{C}{0}{t}{2 \nu}$}, \\
\int_0^t \frac{1}{2}  R_{2 \nu,\nu}(\eta(t) - \eta(\tau),t - \tau) \varphi^-(\tau) \, \d \tau = \delta_{2 \nu}(t)& \text{if $\D{}{}{}{2 \nu} = \D{}{0}{t}{2 \nu}$},
\end{cases}
\end{equation}
while \eqref{right-BC} gives the right BC
\begin{equation*}
\begin{cases}
\int_0^t \frac{1}{2} \delta_{2 \nu}(t - \tau) \varphi^+(\tau) \, \d \tau = -1 & \text{if $\D{}{}{}{2 \nu} = \D{C}{0}{t}{2 \nu}$}, \\
\int_0^t \frac{1}{2} \delta_{2 \nu}(t - \tau) \varphi^+(\tau) \, \d \tau = -1 & \text{if $\D{}{}{}{2 \nu} = \D{}{0}{t}{2 \nu}$}.
\end{cases}
\end{equation*}
Observe that we used Proposition~\ref{R-a-zero}, and both Caputo and Riemann-Liouville cases have the same right BC because they also have the same $h^+(t)$. In fact, the right BC can be expressed as $\D{}{0}{t}{-2 \nu} \varphi^+(t) = -2$ using \eqref{conv-int}. If $\Phi^+(s) = \L\{\varphi^+(t);s\}$, then 
$$
\Phi^+(s) = -\frac{2}{s^{1 - 2 \nu}};
$$ 
thus $\varphi^+(t) = -2 \delta_{1 - 2 \nu}(t)$.

To use the `fractional Stefan condition', we first calculate
\begin{align*}
\frac{\partial u}{\partial x}(x,t) & = -\int_0^t \frac{1}{2} R_{\nu,\nu}(x - \eta(\tau),t - \tau) \varphi^-(\tau)  \, \d \tau,
\end{align*}
which implies that
\begin{equation}
\label{free-prob-2-stefan}
\begin{cases}
r \D{C}{0}{t}{2 \nu} \eta(t) =  1 - \int_0^t \frac{1}{2} R_{\nu,\nu}(\eta(t) - \eta(\tau),t - \tau) \varphi^-(\tau)  \, \d \tau & \text{if $\D{}{}{}{2 \nu} = \D{C}{0}{t}{2 \nu}$}, \\
r \D{}{0}{t}{2 \nu} \eta(t) =  1 - \int_0^t \frac{1}{2} R_{\nu,\nu}(\eta(t) - \eta(\tau),t - \tau) \varphi^-(\tau)  \, \d \tau & \text{if $\D{}{}{}{2 \nu} = \D{}{0}{t}{2 \nu}$}.
\end{cases}
\end{equation}

Hence the solution of the moving boundary problem is described by \eqref{free-prob-2-u}, \eqref{free-prob-2-phi} and \eqref{free-prob-2-stefan}. It does not appear to be possible to solve for $\varphi^-(t)$ and $\eta(t)$ explicitly (assuming that the solutions even exist) and therefore the integral equations have to be solved numerically. Note that although $\varphi^+(t) = -2 \delta_{1 - 2 \nu}(t)$ has been determined for both Caputo and Riemann-Liouville cases, the expressions in \eqref{free-prob-2-u}, \eqref{free-prob-2-phi} and \eqref{free-prob-2-stefan} do not actually depend on it explicitly.
\end{ex}

\section{Concluding remarks}

In this article, we derived the solution of a general IBVP for the time-fractional diffusion equation using the embedding method. The formulation of the IBVP incorporates time-dependent BCs and allows the consideration of bounded and unbounded spatial domains. The solution of the IBVP generalises the results in \citet{RoTh21} for the classical diffusion equation and in \citet{Ro22b} for a particular class of IBVPs with Dirichlet BCs for the time-fractional diffusion equation. We then used the solution of the IBVP to solve two representative examples of moving boundary problems for the time-fractional diffusion equation. In particular, the solutin of the first problem is a `fractional' generalisation of the well-known Neumann solution for a Stefan problem for melting ice. 

The embedding method gives rise to a system of integral equations for some time-dependent functions, and which needs to be solved numerically in general. The numerical solution of IBVPs and moving boundary problems for the time-fractional diffusion equation is currently work in progress. However, the numerical solution of IBVPs for the classical diffusion equation has been done in \citet{RoTh22}. The novelty here for IBVPs for the time-fractional diffusion equation is that the linear Volterra integral equations of the first kind for $\varphi^\pm(t)$ now involve $R_{\mu,\nu}(a,t)$. Hence it is necessary to be able to compute these numerically. As this auxiliary function satisfies certain fractional integral and differential equations, a necessary first step seems to be to solve these equations numerically (e.g.~using finite differences) for $R_{\mu,\nu}(a,t)$ and adapt the boundary element method for solving linear Volterra integral equations of the first kind proposed in \cite{RoTh22} for the classical diffusion equation. Other future directions are multilayer problems for the time-fractional diffusion equation and a further investigation of the properties and applications of the auxiliary function~$R_{\mu,\nu}(a,t)$. 

%

\end{document}